\documentclass[12pt]{article}
\usepackage{a4wide}

\usepackage[a4paper, margin=2.3cm]{geometry}
\usepackage{amsmath,amssymb,amsthm}
\usepackage{color}
\usepackage{parskip}
\usepackage{hyperref}
\usepackage{parskip}
\usepackage{setspace}
\usepackage{graphicx}

\newtheorem{theorem}{Theorem}[section]

\newtheorem{corollary}[theorem]{Corollary}
\newtheorem{lemma}[theorem]{Lemma}

\newtheorem*{definition*}{Definition}

\begin{document}
\title{On a theorem of Hegyv\'{a}ri and Hennecart}
\author{Dao Nguyen Van Anh \thanks{The Olympia Schools Hanoi. Email: dao anh.dnv@theolympiaschools.edu.vn}\and Le Quang Ham\thanks{Department of Mathematics, Hanoi University of Science. Email: hamlaoshi@gmail.com}\and Doowon Koh \thanks{Department of Mathematics, Chungbuk National University. Email: koh131@gmail.com}\and Thang Pham \thanks{Department of Mathematics, University of Rochester New York. Email: vpham2@math.rochester.edu}\and Le Anh Vinh \thanks{Vietnam Institute of Educational Sciences. Email: vinhle@vnies.edu.vn}}
\date{}
\maketitle  

\begin{abstract}
In this paper, we study growth rate of product of sets in the Heisenberg group over finite fields and the complex numbers. More precisely, we will give improvements and extensions of recent results due to Hegyv\'{a}ri and Hennecart (2018). 
\end{abstract}
\section{Introduction}
Let $\mathbb{F}_q$ be an arbitrary finite field, where $q$ is a prime power. Let $\mathbb{F}_p$ be the prime field of order $p$. For an integer $n\ge 1$, the Heisenberg group of degree $n$, denoted by 
$H_n(\mathbb{F}_q),$ is defined by a set of the following matrices:
\[[\mathbf{x}, \mathbf{y}, z] :=  \begin{bmatrix} 1 & \mathbf{x} & z \\
\mathbf{0} & I_n & \mathbf{y^t} \\
0 & \mathbf{0} & 1
\end{bmatrix}
\]
where $\mathbf{x}, \mathbf{y} \in \mathbb{F}_q^n$, $z \in \mathbb{F}_q$, $\mathbf{y^t}$ denotes the column vector of $\mathbf{y}$, and $I_n$ is the $n \times n$ identity matrix. 
For $ A\subset \mathbb{F}_q, E, F\subset \mathbb{F}_q^n$, we define
\[
[E, F, A]:=\{[\mathbf{x}, \mathbf{y}, z]\colon \mathbf{x}\in E, \mathbf{y}\in F,\ z\in A\},
\]
and 
\[
[E, F, A][E, F, A]:=\{[\mathbf{x}, \mathbf{y}, z]\cdot [\mathbf{x'}, \mathbf{y'}, z']\colon [\mathbf{x}, \mathbf{y}, z], [\mathbf{x'}, \mathbf{y'}, z']\in [E, F, A]\},
\]
Over recent years, there is an intensive study on growth rate in the Heisenberg group over finite fields and applications. In \cite{HHH}, Hegyv\'ari and Hennecart proved a structure result for \textit{bricks} in Heisenberg groups. The precise statement is as follows.  
	\begin{theorem}[Hegyv\'ari-Hennecart, \cite{HHH}]\label{hh}
		For every $\varepsilon>0,$ there exists a positive integer $n_0(\epsilon)$ such that for all $n\ge n_0 (\epsilon)$ and any sets  $X_i,Y_i, Z\subset \mathbb{F}_p$, $i\in [n]$,  
		$X = \prod_{i=1}^n X_i \subset \mathbb{F}^n_p$, $Y = \prod_{i=1}^n Y_i \subset \mathbb{F}^n_p$ if we form
		\[[X, Y, Z] = \{ [\mathbf{x}, \mathbf{y}, z] ~:~ \mathbf{x}\in X,\, \mathbf{y}\in Y,\, z\in Z \} \subset H_n(\mathbb{F}_p)\] 
		with 
		\begin{equation}\label{f:HH_bricks_Hn}
		|[X, Y, Z]| > |H_n(\mathbb{F}_p)|^{3/4+\epsilon} \,,
		\end{equation}
		then $[X, Y, X][X, Y, Z]$ contains at least $|[X, Y, Z]|/p$ cosets of $[\mathbf{0}, \mathbf{0}, \mathbb{F}_p]$. 
		\label{t:HH_bricks_Hn}
	\end{theorem}
	It follows from the proof of Theorem \ref{hh} in \cite{HHH} that $\epsilon=O(1/n)$. In a very recent work, Shkredov \cite{sh} improved the relation between $\epsilon$ and $n$ in the following theorem. 
\begin{theorem}[Shkredov, \cite{sh}]
	Let $n\ge 2$ be an even number, and $X_i,Y_i, Z\subset \mathbb{F}_p$, $i\in [n]$,  
	$X = \prod_{i=1}^n X_i \subset \mathbb{F}^n_p$, $Y = \prod_{i=1}^n Y_i \subset \mathbb{F}^n_p$,
	\[[X, Y, Z] = \{ [\mathbf{x}, \mathbf{y}, z] ~:~ \mathbf{x}\in X,\, \mathbf{y}\in Y,\, z\in Z \} \subset H_n(\mathbb{F}_p)\] 
	be sets and $X_i$, $Y_i$ have comparable sizes. Set $\mathcal{X}=\max_{i}|X_i|$ and $\mathcal{Y}=\max_i|Y_i|$. If $|Z| \le \mathcal{X} \mathcal{Y}$, $\mathcal{X} \le |Z| \mathcal{Y}$, $\mathcal{Y} \le |Z| \mathcal{X}$ and   
	\begin{equation}\label{f:bricks_Hn}
	\mathcal{X} \mathcal{Y} \gtrsim p^{3/2} \cdot \left(\frac{\mathcal{X} \mathcal{Y}}{p|Z|^{1/2}} \right)^{2^{-{n/2}}} \,,
	\end{equation}
	then $[X, Y, Z][X, Y, Z]$ contains at least  
	$|[X, Y, Z]|/p$ cosets of 
	$[\mathbf{0},\mathbf{0}, \mathbb{F}_p]$. 
	\label{t:bricks_Hn}
\end{theorem}
Moreover, the work of Shkredov \cite{sh} gives an introduction to
representation theory which is good for products of general sets in
the affine and in the Heisenberg groups.

Throughout this paper, we use $X\ll Y$ if $X\le C Y$ for some constant $C>0$ independent of the parameters related to $X$ and $Y,$ and  write $X\gg Y$ for $Y\ll X.$ The notation $X\sim Y$ means that both $X\ll Y$ and $Y\ll X$ hold. In addition, we use $X\lesssim Y$ to indicate that $X\ll (\log{Y})^{C'} Y$ for some constant $C'>0$.

It is worth noting that there is an interesting application of products of sets in the Heisenberg group to
so--called models of \textit{Freiman isomorphisms}, see \cite{h3}. Moreover,  it has been indicated in \cite[Section 5.3]{tao} that for any set in the Heisenberg group with the doubling constant less than two does not have any good model. 

It is well-known that there is a connection between sum--product phenomenon and growth in the group of affine transformations, for example, see \cite{rr2}. Such a connection has been discovered in the setting of Heisenberg group by Hegyv\'{a}ri and Hennecart \cite{HH}. More precisely, in the case $n=1$,  using sum-product estimates, they proved that if $A\subset \mathbb{F}_p$ with $|A|\ge p^{1/2}$, then 
\begin{equation}\label{manhhon}|[A, A, 0] [A, A, 0]|\gg \min\left\lbrace p^{1/2}|[A, A, 0]|^{5/4}, p^{-1/2}|[A, A, 0]|^2  \right\rbrace.\end{equation}
When the size of $A$ is not too big, they obtained the following. 
\begin{theorem}[Hegyv\'ari-Hennecart, \cite{HH}]\label{eq0901}
Let $A$ be a set in $\mathbb{F}_p$. Suppose that $|A|\le p^{2/3}$, then we have 
\[|[A, A, 0] [A, A, 0]|\gg |[A, A, 0]|^{\frac{7}{4}}.\]
\end{theorem}
It is not hard to see that the method in the proof of Theorem \ref{eq0901} can be extended to arbitrary finite fields, and as a consequence, we obtain the following. 
\begin{theorem}[Hegyv\'ari-Hennecart, \cite{HH}]\label{thmcuaHH}
Let $A$ be a set in $\mathbb{F}_q$. Suppose that $|A|\ge q^{2/3}$, then we have 
\[|[A, A, 0][A, A, 0]|\gg q|[A, A, 0]|.\]
\end{theorem}
Note that the lower bound in Theorem \ref{thmcuaHH} is stronger than that of (\ref{manhhon}). 

The main purpose of this paper is to give improvements and extensions of Theorems \ref{eq0901} and \ref{thmcuaHH} in the setting of arbitrary finite fields $\mathbb{F}_q$ and the complex numbers $\mathbb{C}$. 

In our first theorem, we will show that Theorem \ref{thmcuaHH} can be improved in the case the additive energy of $A$ is small. 
\begin{theorem}\label{thm0}
Let $A$ be a set in $\mathbb{F}_q$. Let $E^+(A)$ be the number of quadruples $(a, b, c, d)\in A^4$ such that $a+b=c+d$. Suppose that $E^+(A)\le \frac{|A|^3}{K}$ for some $K>0$ and $|A|\ge K^{1/3}q^{2/3}$,
then we have 
\[|[A, A, 0][A, A, 0]|\gg Kq|[A, A, 0]|.\]
\end{theorem}
Our next theorem is an extension of Theorem \ref{thmcuaHH} in the setting of $H_n(\mathbb{F}_q)$ for any $n\ge 1$.
\begin{theorem}\label{thm1}
Let $E$ be a set in $\mathbb{F}_q^n$. Suppose that $|E|\gg q^{\frac{n}{2}+\frac{1}{4}}$, then we have 
\[|[E, E, 0][E, E, 0]|\gg q|[E, E, 0]|.\]
\end{theorem}

Notice that in general the conclusion of Theorem \ref{thm1} is sharp, since $E$ can be a subspace in $\mathbb{F}_q^n$, which implies that $[E, E, 0][E, E, 0]\subset [E, E, \mathbb{F}_q]$. Moreover, the exponent $\frac{n}{2}+\frac{1}{4}$ can not be decreased to $\frac{n}{2}$, since, suppose that $q=p^2$, then one can take $E=\mathbb{F}_p^n$, which gives us  $|[E, E, 0][E, E, 0]|\ll p|[E, E, 0]|=q^{1/2}|[E, E, 0]|.$

In the setting of prime fields, if $E$ is a set in the plane $\mathbb{F}_p^2$ and the size of $E$ is not too big, then we have the following theorem in $H_2(\mathbb{F}_p)$.
\bigskip
\begin{theorem}\label{thm2}
Let $\mathbb{F}_p$ be a prime field with $p\equiv 3\mod 4$, and $E$ be a set in $\mathbb{F}_p^2$ with $|E|\ll p^{8/5}$. Then
\[|[E, E, 0][E, E, 0]|\gg |[E, E, 0]|^{\frac{19}{15}}.\]
\end{theorem}

When $A$ is a multiplicative subgroup of $\mathbb{F}_p^*$, we are able to show that the exponent $\frac{7}{4}$ in  Theorem \ref{eq0901} can be improved significantly. 
\begin{theorem}\label{thm3}
Let $A$ be a multiplicative subgroup of $\mathbb{F}_p^*$ with $|A|\le p^{1/2}\log (p)$. We have 
\[|[A, A, 0][A, A, 0]|\gtrsim |[A, A, 0]|^{\frac{151}{80}}.\]
\end{theorem}

In the setting of the real numbers, for any $A\subset \mathbb{R}$, Shkredov \cite{sh} recently proved that
\[|[A, A, 0][A, A, 0]|\gg |[A, A, 0]|^{\frac{7}{4}+c},\]
for some small $c>0$. This improves an earlier result given by Hegyv\'{a}ri and Hennecart \cite{HH}.  In our next theorem, we give a further improvement and extend it to the setting of the complex numbers. 
\begin{theorem}\label{thm5}
Let $A$ be a set in $\mathbb{C}$ with $|A|\ge 2$. We have 
\[|[A, A, 0][A, A, 0]|\gtrsim |A|^{\frac{29}{8}}=|[A, A, 0]|^{\frac{29}{16}}.\]
\end{theorem}
\section*{Acknowledgments}
Doowon Koh was supported by the National Research Foundation of Korea (NRF) grant funded by the Korea government (MIST) (No. NRF-2018R1D1A1B07044469).  Thang Pham was supported by Swiss National Science Foundation grant P400P2-183916.
\section{Proof of Theorem \ref{thm0}}
To prove Theorem \ref{thm0}, we need to recall a lemma given by the third, fourth, fifth listed authors in 
\cite{kpv}. 

Let $X$ be a multi-set in $\mathbb{F}_q^{2n}\times \mathbb{F}_q$. We denote by $\overline{X}$ the set of distinct elements in the multi-set $X$.
The cardinality of $X$, denoted by $|X|$, is $\sum_{x\in \overline{X}}m_X(x)$, where $m_X(x)$ is the multiplicity of $\mathbf{x}$ in $X$. For multi-sets $\mathcal{A}, \mathcal{B}\subset \mathbb{F}_q^{2n+1}$, let $N(\mathcal{A}, \mathcal{B})$ be the number of pairs $\left((\mathbf{a}, b), (\mathbf{c}, d)\right)\in \mathcal{A}\times \mathcal{B}\subset \left(\mathbb{F}_q^{2n}\times \mathbb{F}_q\right)^2$ such that $\mathbf{a}\cdot \mathbf{c}=b+d$. We have the following lemma on an upper bound of $N(\mathcal{A}, \mathcal{B})$. 

\bigskip

\begin{lemma}[\cite{kpv}, Lemma 8.1]\label{fourier}
Let $\mathcal{A}, \mathcal{B}$ be a multi-sets in $\mathbb{F}_q^{2n}\times \mathbb{F}_q$. We have
\[\left\vert N(\mathcal{A}, \mathcal{B})-\frac{|\mathcal{A}||\mathcal{B}|}{q}\right\vert \le q^{n}\left(\sum_{(a, b)\in \overline{\mathcal{A}}}m_\mathcal{A}((a, b))^2\sum_{(c, d)\in \overline{\mathcal{B}}}m_\mathcal{B}((c, d))^2\right)^{1/2}.\]
\end{lemma}
\bigskip
Theorem \ref{thm0} is a direct consequence of the following theorem.
\bigskip
\begin{theorem}
For $A\subset \mathbb{F}_p$, we have 
\[|[A, A, 0][A, A, 0]|\gg \min \left\lbrace \frac{|A|^5}{q}, \frac{q|A|^5}{E^+(A)}\right\rbrace.\]
\end{theorem}
\begin{proof}
Without loss of generality, we assume that $0\not\in A$. Let $S$ be the number of quadruples of matrices $(m_1, m_2, m_3, m_4)$ in $[A, A, 0]^4$ such that $m_1m_2=m_3m_4$. By the Cauchy-Schwarz inequality, we have 
\[|[A,A, 0]^2|\ge \frac{|A|^8}{S}.\]
In the next step, we are going to show that 
\[S\ll \frac{|A|^3E^+(A)}{q}+q|A|^3.\]
Indeed, it is not hard to check that $S$ is equal to the number of tuples $(a, b, c, d, a', b', c', d')$ in $A^8$ such that 
\begin{align}
\label{Nsystem11} &a+c=a'+c',\\
\label{Nsystem21} &b+d=b'+d'\\
\label{Nsystem31} &ad=a'd'
\end{align}
It follows from (\ref{Nsystem11}) and (\ref{Nsystem21}) that $a=a'+c'-c$ and $d'=b+d-b'$. Substituting into (\ref{Nsystem31}), we obtain 
\[(a'+c'-c)\cdot d=a'\cdot (b+d-b').\]
This implies that 
\begin{equation}\label{mat1}d(c'-c)=a'(b-b').\end{equation}
This is equivalent with 
\[a'=\frac{d}{b-b'}(c'-c).\]
It follows from (\ref{mat1}) that if $b=b'$ then $c=c'$. We note that the number of tuples $(a', b, b', c, c', d)\in A^6$ with $b=b'$ and $c=c'$ is at most $|A|^4$.  We now count the number of tuples with $b\ne b'$ and $c\ne c'$. It is not hard to check the number of tuples $(a, b, c, d, a', b', c', d')\in A^8$ satisfying (\ref{Nsystem11}-\ref{Nsystem31}) is at most the number of tuples $(a', c, c', b, b', d)\in A^6$ with 
\[a'=\frac{d}{b-b'}(c'-c),\]
where $a', c'\in A$ and $b+d-b'\in A$. Let $X$ be the number of such tuples. 

It is not hard to check that $X$ is bounded by the number of incidences between the point set $P=A\times A$ and the multi-set $L$ of lines of the from $y=\frac{d}{b-b'}(x-c)$ with $b+d-b'\in A$. It is clear that $|P|=|A|^2$ and $|L|=E^+(A)|A|$.  

Let $\mathcal{L}$ be the multi-set in $\mathbb{F}_q^2$ containing points of the form $(\frac{d}{b-b'}, \frac{d}{b-b'}\cdot c)$ with $b+d-b'\in A$ and $c\in A$. On the other hand, by an elementary calculation, we have 
$\sum_{l\in \overline{\mathcal{L}}}m(l)^2\le X|A|$, and $|\mathcal{L}|=|L|$. It is not hard to check that  $X=N(P, \mathcal{L})$, where $N(P, \mathcal{L})$ is defined as in Lemma \ref{fourier}. Applying Lemma \ref{fourier}, we have 
\[X\le \frac{|A|^3E^+(A)}{q}+q^{1/2}X^{1/2}|A|^{3/2},\]
which implies that 
\[X\le \frac{|A|^3E^+(A)}{q}+q|A|^3.\]
In other words, we have 
\[S\le \frac{|A|^3E^+(A)}{q}+q|A|^3+|A|^4\ll \frac{|A|^3E^+(A)}{q}+q|A|^3.\]
\end{proof}
\section{Proof of Theorem \ref{thm1}}
In order to prove Theorem \ref{thm1}, we first prove the following lemma. 
\begin{lemma}\label{lm2}
Let $E$ be a set in $\mathbb{F}_q^n$. Let $T$ be the number of triples $(\mathbf{v}, \mathbf{x}, \mathbf{x}')\in E^3$ such that $\mathbf{v}\cdot (\mathbf{x}-\mathbf{x}')=0$. Then we have 
\[T\le \frac{|E|^3}{q}+q^n|E|.\]
\end{lemma}
Before proving Lemma  \ref{lm2}, we need to review the Fourier transform of functions on $\mathbb{F}_q^n$. Let $\chi$ be a non-trivial additive character on $\mathbb{F}_q$. For a function $f: \mathbb{F}_q ^n\to \mathbb{C}$, we define 
\[\widehat{f}(\mathbf{m})=q^{-n} \sum_{\mathbf{x} \in \mathbb{F}_q^n} \chi(-\mathbf{x} \cdot \mathbf{m}) f(\mathbf{x}).\]
It is not hard to see that \[ f(\mathbf{x})=\sum_{\mathbf{m} \in \mathbb{F}_q^n} \chi(\mathbf{x} \cdot \mathbf{m}) \widehat{f}(\mathbf{m}),\] and
\[ \sum_{\mathbf{m} \in \mathbb{F}_q^n} {|\widehat{f}(\mathbf{m})|}^2=q^{-n} \sum_{\mathbf{x} \in \mathbb{F}_q^n} {|f(\mathbf{x})|}^2.\]

We are now ready to prove Lemma \ref{lm2}.
\paragraph{Proof of Lemma \ref{lm2}:} The number $T$ can be expressed as follows:
\begin{align}\label{keymoment1}
T=\sum_{\mathbf{x} \cdot \mathbf{v}-\mathbf{x}' \cdot \mathbf{v}=0} E(\mathbf{v})E(\mathbf{x})E(\mathbf{x}')
&=\frac{|E|^3}{q}+\frac{1}{q}\sum_{s \not=0} \sum_{\mathbf{v}, \mathbf{x}, \mathbf{x}'} \chi(s\mathbf{v} \cdot (\mathbf{x}-\mathbf{x}')) E(\mathbf{v}) E(\mathbf{x})E(\mathbf{x}')\nonumber\\
&= \frac{|E|^3}{q}+q^{2n-1} \sum_{s \not=0} \sum_{\mathbf{v}} {|\widehat{E}(s\mathbf{v})|}^2 E(\mathbf{v})\nonumber\\
&\leq \frac{|E|^3}{q}+q^{2n}  \sum_{\mathbf{z} \in \mathbb{F}_q^n} {|\widehat{E}(\mathbf{z})|}^2\nonumber\\
&=\frac{|E|^3}{q}+q^{n}|E|.
\end{align}
where we used $\sum_{\mathbf{z}\in \mathbb{F}_q^n}{|\widehat{E}(\mathbf{z})|}^2=q^{-n}|E|.$ This completes the proof of the lemma. $\hfill\square$

We are ready to prove Theorem \ref{thm1}.
\paragraph{Proof of Theorem \ref{thm1}:}
Let $S$ be the number of quadruples of matrices $(m_1, m_2, m_3, m_4)$ in $[E, E, 0]^4$ such that $m_1m_2=m_3m_4$. By the Cauchy-Schwarz inequality, we have 
\[|[E,E, 0][E, E, 0]|\ge \frac{|E|^8}{S}.\]
In the next step, we are going to show that 
\[S\le \frac{|E|^6}{q}+q^{n-1}|E|^4+q^{2n}|E|^2.\]
Indeed, it is not hard to check that $S$ is equal to the number of tuples $(\mathbf{a}, \mathbf{b}, \mathbf{c}, \mathbf{d}, \mathbf{a'}, \mathbf{b'}, \mathbf{c'}, \mathbf{d'})$ in $E^8$ such that 
\begin{align}
\label{Nsystem1} &\mathbf{a}+\mathbf{c}=\mathbf{a'}+\mathbf{c}',\\
\label{Nsystem2} &\mathbf{b}+\mathbf{d}=\mathbf{b'}+\mathbf{d'}\\
\label{Nsystem3} &\mathbf{a}\cdot \mathbf{d}=\mathbf{a}'\cdot \mathbf{d}'
\end{align}
It follows from (\ref{Nsystem1}) and (\ref{Nsystem2}) that $\mathbf{a}=\mathbf{a'}+\mathbf{c'}-\mathbf{c}$ and $\mathbf{d'}=\mathbf{b}+\mathbf{d}-\mathbf{b'}$. Substituting into (\ref{Nsystem3}), we obtain 
\[(\mathbf{a'}+\mathbf{c}'-\mathbf{c})\cdot \mathbf{d}=\mathbf{a}'\cdot (\mathbf{b}+\mathbf{d}-\mathbf{b'}).\]
This implies that 
\begin{equation}\label{mat}\mathbf{d}\cdot (\mathbf{c'}-\mathbf{c})=\mathbf{a'}\cdot (\mathbf{b}-\mathbf{b'}).\end{equation}
For any tuples $(\mathbf{c}, \mathbf{c'}, \mathbf{b}, \mathbf{b'}, \mathbf{d}, \mathbf{a'})$ satisfying (\ref{mat}), we have $\mathbf{a}$ and $\mathbf{d'}$ are determined uniquely by (\ref{Nsystem1}) and (\ref{Nsystem2}). 

Let $\mathcal{A}$ and $\mathcal{B}$ be multisets defined as follows: 
\[\mathcal{A}=\{(\mathbf{d}, -\mathbf{b}, \mathbf{d}\cdot \mathbf{c})\colon \mathbf{b}, \mathbf{c}, \mathbf{d}\in E\}, ~\mathcal{B}=\{(\mathbf{c'}, \mathbf{a'}, -\mathbf{a'}\cdot \mathbf{b'})\colon \mathbf{a'}, \mathbf{b'}, \mathbf{c'}\in E\}.\]
Let $N(\mathcal{A}, \mathcal{B})$ the number defined as in Lemma \ref{fourier}. We have that the number of tuples satisfying (\ref{mat}) is equal to $N(\mathcal{A}, \mathcal{B})$. 

In order to apply Lemma \ref{fourier}, we need to estimate $\sum_{\mathbf{x}\in \overline{\mathcal{A}}}m_{\mathcal{A}}(\mathbf{x})^2$ and $\sum_{\mathbf{y}\in \overline{\mathcal{B}}}m_{\mathcal{B}}(\mathbf{y})^2$. 

By an elementary calculation, we have 
\[\sum_{\mathbf{x}\in \overline{A}}m_{\mathcal{A}}(\mathbf{x})^2, ~\sum_{\mathbf{y}\in \overline{\mathcal{B}}}m_{\mathcal{B}}(\mathbf{y})^2\le |E|T,\]
where $T$ is the number of triples $(\mathbf{v}, \mathbf{x}, \mathbf{x}')\in E^3$ such that $\mathbf{v}\cdot (\mathbf{x}-\mathbf{x}')=0$. 

On the other hand, Lemma \ref{lm2} gives us 
\[T\le \frac{|E|^3}{q}+q^{n}|E|.\]
Therefore, one can apply Lemma \ref{fourier} with $|\mathcal{A}|=|\mathcal{B}|=|E|^3$ to derive 
\[S\le \frac{|E|^6}{q}+q^n\left(\frac{|E|^4}{q}+q^n|E|^2\right)\ll \frac{|E|^6}{q},\]
whenever $|E|\gg q^{\frac{2n+1}{4}}$. This concludes the proof of the theorem. $\hfill\square$. 
\section{Proof of Theorem \ref{thm2}}
To prove Theorem \ref{thm2}, we need to use the following lemmas. The first lemma is a point-line incidence bound due to Stevens and De Zeeuw in \cite{lund}. 
\bigskip
\begin{lemma}\label{frank}
Let $P$ be a point set in $\mathbb{F}_p^2$ and $L$ be a set of lines in $\mathbb{F}_p^2$. Suppose that $|P|\le p^{8/5}$, then the number of incidences between $P$ and $L$, denoted by $I(P, L)$, satisfying 
\[I(P, L)\ll |P|^{11/15}|L|^{11/15}+|P|+|L|.\]
\end{lemma}
\begin{lemma}\label{lm3}
Let $E$ be a set in $\mathbb{F}_p^2$ with $p\equiv 3\mod 4$ and $|E|\le p^{8/5}$. We have $|\Pi(E)|\gg |E|^{8/15}$. 
\end{lemma}
\begin{proof}
Since $p\equiv 3\mod 4$, there is no isotropic line in $\mathbb{F}_p^2$. For each $a\in E$, we denote the set $\{a\cdot b\colon b\in E\}$ by $\Pi_a(E)$. Suppose that 
\[\max_{a\in E}|\Pi_a(E)|=t.\]
It is clear that $|\Pi(E)|\gg \max_{a\in E}|\Pi_a(E)|$.

Without loss of generality, we may assume that $0\notin E.$ We now fall into two following cases:

{\bf Case $1$:} If there is a line passing through the origin with at least $m$ points of $E$, then it is not hard to check that $|\Pi(E)|\gg m$. 

{\bf Case $2$:} Suppose that all lines passing through the origin contain at most $m$ points of $E.$ This implies that the number of lines passing through the origin and a point in $E$ is at least $|E|/m$.

Let $L_0$ be a set of lines passing through the origin and at least one point from $E$ such that $|L_0|\sim |E|/m$. From each line $l$ in $L_0$, we pick one point in $l\cap E$ arbitrary, and let $P$ be the set of those points. So $|P|=|L_0|$. 

For any point $a=(a_1, a_2)\in E$, let $L_a$ be the set of lines defined by the equation $a_1x+a_2y=r$ with $r\in \Pi_a(E)$. One can check that the size of $L_a$ is the same as the size of $\Pi_a(E)$. More  over, one can check that $L_a=L_b$ when both $a$ and $b$ lie on a line in $L_0$, and $L_a\cap L_b=\emptyset$ when the $a$ and $b$ are distinct elements of $P.$ 

Let $L=\cup_{a\in P}L_a$. Since $|\Pi_a(E)|\le t$ for any $a\in E$, we have $|L_a|\le t$ for all $ a\in E$. Thus $|L|\le |P|t=|L_0|t\sim|E|t/m$. 

Let $I(E, L)$ be the number of incidences between $E$ and $L$. For each $a\in P$, we have $I(E, L_a)=|E|$. Thus, 
\[I(E, L)\gg |E|^2/m.\]

On the other hand, it follows from Lemma \ref{frank} that 
\[I(E, L)\ll |E|^{11/15}(|E|t/m)^{11/15}+|E|+|E|t/m.\]

Hence, we have
$$ |E|^2/m \ll |E|^{11/15}(|E|t/m)^{11/15}+|E|+|E|t/m.$$
Since $|E|^2/m \gg |E|+|E|t/m,$ 
solving this inequality for $t$, we obtain $t\gg |E|^{8/11}m^{-4/11}$.

Optimizing two cases by choosing $m=|E|^{8/15},$ the lemma follows. 
\end{proof}
\paragraph{Proof of Theorem \ref{thm2}:}
We first observe that 
\[|[E, E, 0][E, E, 0]|\gg |\Pi(E)||E|^2.\]
It follows from Lemma \ref{lm3} that if $|E|\le p^{8/15}$ then we have 
\[|\Pi(E)|\gg |E|^{8/15}.\]
Therefore, 
\[|[E, E, 0][E, E, 0]|\gg |\Pi(E)||E|^2\gg |E|^{38/15},\]
whenever $|E|\ll p^{8/15}$. This completes the proof of the theorem. $\hfill\square$. 
\section{Proof of Theorem \ref{thm3}}
In the proof of Theorem \ref{thm3}, the following results will be used. 
\bigskip
\begin{lemma}
Let $A$ be a multiplicative subgroup of $\mathbb{F}_p^*$ with $|A|\lesssim p^{1/2}$. Let $L$ be a set of lines in $\mathbb{F}_p^2$, and $I(A\times A, L)$ be the number of incidences between $A\times A$ and $L$. We have 
\[I(A\times A, L)\lesssim |A|^{4/3}|L|^{2/3}.\]
\end{lemma}
\begin{proof}
Let $T(A)$ be the number of collinear triples of points in $A\times A$. It has been shown in \cite[Theorem 1.2]{mac} that if $|A|\lesssim p^{1/2}$, then we have 
\[T(A)\lesssim |A|^4.\]

For any $l\in L$, let $i(l)$ be the number of points of $A\times A$ on $l$. We have 
\[I(A\times A, L)=\sum_{l\in L}i(l)\le |L|^{2/3}\left(\sum_{l\in L}i(l)^3\right)^{1/3}=|L|^{2/3}T(A)^{1/3},\]
where we used the Cauchy-Schwarz inequality in the inequality step. 

Since $T(A)\lesssim |A|^4$, the lemma follows. 
\end{proof}
The following theorem is given in \cite[Theorem $3$]{s}. 
\bigskip
\begin{theorem}\label{ss}
Let $A$ be a multiplicative subgroup of $\mathbb{F}_p^*$. Suppose that $|A|\le p^{1/2}$, then we have 
\[E^+(A)\lesssim |A|^{49/20}.\]
\end{theorem}
We are now ready to prove Theorem \ref{thm3}.
 \paragraph{Proof of Theorem \ref{thm3}:} We first repeat the first paragraph in the proof of Theorem \ref{thm0}. 
 
 Let $S$ be the number of quadruples of matrices $(m_1, m_2, m_3, m_4)$ in $[A, A, 0]^4$ such that $m_1m_2=m_3m_4$. By the Cauchy-Schwarz inequality, we have 
\[[|A,A, 0][A, A, 0]|\ge \frac{|A|^8}{S}.\]
Thus, to complete the proof, we only to show that 
\[S\lesssim  |A|^4+|A|^{\frac{169}{40}}.\]
Indeed, it is not hard to check that $S$ is equal to the number of tuples $(a, b, c, d, a', b', c', d')$ in $A^8$ such that 
\begin{align}
\label{Nsystem111} &a+c=a'+c',\\
\label{Nsystem211} &b+d=b'+d'\\
\label{Nsystem311} &ad=a'd'
\end{align}
It follows from (\ref{Nsystem111}) and (\ref{Nsystem211}) that $a=a'+c'-c$ and $d'=b+d-b'$. Substituting into (\ref{Nsystem311}), we obtain 
\[(a'+c'-c)\cdot d=a'\cdot (b+d-b').\]
This implies that 
\begin{equation}\label{mat11}d(c'-c)=a'(b-b').\end{equation}
This is equivalent with 
\[a'=\frac{d}{b-b'}(c'-c).\]
It follows from (\ref{mat11}) that if $b=b'$ then $c=c'$. We note that the number of tuples $(a', b, b', c, c', d)\in A^6$ with $b=b'$ and $c=c'$ is at most $|A|^4$.  We now count the number of tuples with $b\ne b'$ and $c\ne c'$. We have the number of tuples $(a, b, c, d, a', b', c', d')\in A^8$ satisfying (\ref{Nsystem111}-\ref{Nsystem311}) is at most the number of tuples $(a', c, c', b, b', d)\in A^6$ with 
\[a'=\frac{d}{b-b'}(c'-c),\]
where $a', c'\in A$ and $b+d-b'\in A$. Let $X$ be the number of such tuples. So, $S\le X+|A|^4$.

On the other hand, $X$ is bounded by the number of incidences between the point set $P=A\times A$ and the multi-set $L$ of lines of the from $y=\frac{d}{b-b'}(x-c)$ with $b+d-b'\in A$. It is clear that $|P|=|A|^2$ and $|L|=E^+(A)|A|$. For any line $l\in L$, let  $m(l)$ be the multiplicity of $l$. By an elementary calculation, we have
\[\sum_{l\in \overline{L}}m(l)^2\le X|A|.\]

Let $L_k$ be the set of lines $l\in L$ (without multiplicity) with $k\le m(l)\le 2k$. For any $k$, we have 
\[k|L_k|\le |L|=E^+(A)|A|, ~k^2|L_k|\le \sum_{l\in \overline{L}}m(l)^2\le X|A|.\]
We have 
\begin{align*}
I(P, L)&\le \sum_{\mathtt{dyadic} ~k}2k\cdot I(P, L_k)=\sum_{\mathtt{dyadic} ~k}2k\cdot I(P, L_k)\\
&\le \sum_{\mathtt{dyadic} ~k \le \frac{X}{E^+(A)}}2k\cdot I(P, L_k)+\sum_{\mathtt{dyadic} ~k \ge \frac{X}{E^+(A)}}2k\cdot I(P, L_k)\\
&\lesssim \sum_{\mathtt{dyadic} ~k \le \frac{X}{E^+(A)}}2k\cdot |A|^{4/3}\left(\frac{E^+(A)|A|}{k}\right)^{2/3}+\sum_{\mathtt{dyadic} ~k \ge \frac{X}{E^+(A)}}2k\cdot |A|^{4/3}\left(\frac{X|A|}{k^2}\right)^{2/3}\\
&\lesssim |A|^2X^{1/3}E^+(A)^{1/3}.
\end{align*}
In other words, we have proved that 
\[X\lesssim |A|^2X^{1/3}E^+(A)^{1/3}, \]
which implies that $X\lesssim |A|^3E^+(A)^{1/2}$. 
Applying Theorem \ref{ss}, we have 
\[X\lesssim |A|^{\frac{169}{40}},\]
whenever $|A|\lesssim p^{1/2}$. This completes the proof of the theorem. $\hfill\square$

\section{Proof of Theorem \ref{thm5}}
The proof of Theorem \ref{thm5} is quite similar compared to that of Theorem \ref{thm3}. More precisely, we will need the following point-line incidence bound over the complex numbers due to 
 T\'{o}th in \cite{toth}. 
\begin{theorem}[\cite{toth}]\label{incidence-C}
Let $P$ be a set of points in $\mathbb{C}^2$ and $L$ be a set of lines in $\mathbb{C}^2$. The number of incidences between $P$ and $L$, denoted by $I(P, L)$, satisfies
\[I(P, L)\ll |P|^{2/3}|L|^{2/3}+|P|+|L|.\]
\end{theorem}
\begin{corollary}[\cite{toth}]\label{diemgiau}
Let $P$ be a set of points in $\mathbb{C}^2$. For any integer $t\ge 2$, the number of lines containing at least $t$ points from $P$ is bounded by 
\[O\left(\frac{|P|^2}{t^3}+\frac{|P|}{t}\right).\]
\end{corollary}
Using these results, we have the following corollary. 
\begin{corollary}\label{cococo}
Let $A$ be a set in $\mathbb{C}$. Let $T(A)$ be the number of collinear points in $A\times A$, we have 
\[T(A)\lesssim |A|^4.\]
\end{corollary}
\begin{proof}
Let $L_k$ be the set of lines $l$ such that $2^k\le |l\cap (A\times A)|<2^{k+1}$. Since $|l\cap (A\times A)|\le |A|$ for any $l$, we have $k\ll \log (|A|)$. Thus, using Corollary \ref{diemgiau}, we have
\begin{align*}
T(A)&=\sum_{k} \sum_{l\in L_k} |l\cap (A\times A)|^3\\
&=\sum_{k} \left( \frac{|A|^4}{2^{3k}}+\frac{|A|^2}{2^{k}}\right)\cdot 2^{3k+3}\\
&\lesssim |A|^4+\sum_{k}|A|^22^{2k}\lesssim |A|^4,
\end{align*}
where we have used the fact that $2^k\le |A|$. 
\end{proof}
\begin{lemma}\label{thmx-x9}
Let $A$ be a set in $\mathbb{C}$ with $|A|\ge 2$. The number of tuples $(a, b, c, a', b', c')\in A^6$ such that 
\[a(b-c)=a'(b'-c')\]
is bounded by $E^\times (A)^{1/2}|A|^3+|A|^4\le 2E^\times (A)^{1/2}|A|^3$. 
\end{lemma}
\begin{proof}
Since $|A|\ge 2$, without loss of generality, we assume that $0\not\in A$. We first have an observation that the number of desired tuples with $b=c$ or $b'=c'$ is at most $|A|^4\le E^\times (A)^{1/2}|A|^3$ since $E^\times (A)\ge |A|^2$. 

Let $M$ be the number of tuples with $b\ne c$ and $b'\ne c'$. We have $M$ is equal to the number of desired tuples $(a, b, c, a', b', c')\in A^6$ such that 
\[\frac{a}{a'}=\frac{b'-c'}{b-c}.\]
Using the Cauchy-Schwarz inequality, we have 
\[M\le E^\times (A)^{1/2}\cdot \left\vert\left\lbrace (b_1, c_1, b_2, c_2, b_3, c_3, b_4, c_4)\in A^8\colon \frac{b_1-c_1}{b_2-c_2}=\frac{b_3-c_3}{b_4-c_4}\right\rbrace \right\vert^{1/2}.\]
Using the Cauchy-Schwarz inequality one more time, we have 
\begin{align*}
&\left\vert\left\lbrace (b_1, c_1, b_2, c_2, b_3, c_3, b_4, c_4)\in A^8\colon \frac{b_1-c_1}{b_2-c_2}=\frac{b_3-c_3}{b_4-c_4}\right\rbrace \right\vert\le |A|^2 \cdot \\ &\times \left\vert\left\lbrace (b_1, c_1, b_2, c_2, d_1, d_2)\in A^6\colon \frac{b_1-c_1}{b_2-c_2}=\frac{d_1-c_1}{d_2-c_2}\right\rbrace \right\vert \\
&\le |A|^2\cdot T(A)\lesssim |A|^6,
\end{align*}
where we have used the Corollary \ref{cococo} in the last inequality.
\end{proof}
\paragraph{Proof of Theorem \ref{thm5}:}
Without loss of generality, we assume that $0\not\in A$. It has been proved in \cite{iliya} that there exist $B, C\subset A$ such that $|B|, |C|\ge |A|/3$ and 
\[E^+(B)\cdot E^\times (C)\lesssim |A|^{11/2}.\]
This implies that $E^+(B)\lesssim |A|^{11/4}$ or $E^\times (C)\lesssim |A|^{11/4}$. If $E^+(B)\lesssim |A|^{11/4}$ then we replace the set $A$ in the Theorem \ref{thm5} by $B$, otherwise, we replace the set $A$ by $C$. Thus, we may assume that either $E^+(A)\lesssim |A|^{11/4}$ or $E^\times (A)\lesssim |A|^{11/4}$

The rest of proof of Theorem \ref{thm5} is almost identical with that of Theorem \ref{thm3}, and the last step is to estimate $X$. 

Using Theorem \ref{incidence-C} and the same argument as in the proof of Theorem \ref{thm3}, we have \[X\lesssim |A|^3E^+(A)^{1/2}.\]

On the other hand, using Lemma \ref{thmx-x9}, we have 
\[X\lesssim |A|^3E^\times (A)^{1/2}.\]

Since either $E^+(A)\lesssim |A|^{11/4}$ or $E^\times (A)\lesssim |A|^{11/4}$, we have 
\[X\lesssim |A|^{3+\frac{11}{8}}.\]
Therefore, 
\[|[A, A, 0][A, A, 0]|\gg |A|^{5-\frac{11}{8}}=|A|^{\frac{29}{8}}=|[A, A, 0]|^{\frac{29}{16}}.\]
This completes the proof of the theorem. $\hfill\square$

\end{document}